\documentclass[11pt,letterpaper]{amsart}
\pdfoutput=1

\usepackage[mathscr]{eucal}
\usepackage[usenames,dvipsnames]{xcolor}
\usepackage{amsmath}
\usepackage{amssymb}
\usepackage{amsthm}
\usepackage{array}
\usepackage{bbold}
\usepackage{enumerate}
\usepackage{etoolbox}
\usepackage{stmaryrd}
\usepackage{xspace}
\usepackage{xr-hyper}
\usepackage[colorlinks=true,linkcolor={Brown},citecolor={Brown},urlcolor={Brown}]{hyperref}

\usepackage[all]{xy}
\SelectTips{cm}{}
\newdir{ >}{{}*!/-10pt/\dir{>}}

\usepackage{cleveref}

\crefname{Thm}{Theorem}{Theorems}
\crefname{Rem}{Remark}{Remarks}
\crefname{Prop}{Proposition}{Propositions}

\numberwithin{equation}{section}
\swapnumbers 
\newtheorem{Cor}[equation]{Corollary}
\newtheorem{Lem}[equation]{Lemma}
\newtheorem{Prop}[equation]{Proposition}
\newtheorem{Thm}[equation]{Theorem}
\theoremstyle{remark}
\newtheorem{Def}[equation]{Definition}
\newtheorem{Not}[equation]{Notation}
\newtheorem{Exa}[equation]{Example}
\newtheorem{Rem}[equation]{Remark}
\newtheorem{Sch}[equation]{Scholium}
\newtheorem{Rec}[equation]{Recollection}
\newtheorem*{Ack}{Acknowledgements}
\newcommand{\nc}{\newcommand}
\nc{\dmo}{\DeclareMathOperator}

\dmo{\Chain}{Ch}
\dmo{\chara}{char}%
\dmo{\cone}{cone}
\dmo{\Der}{D}
\dmo{\Gl}{Gl}
\dmo{\Hom}{Hom}
\dmo{\id}{id}
\dmo{\Img}{Im}
\dmo{\im}{im}
\dmo{\incl}{incl}
\dmo{\Ind}{Ind}
\dmo{\Infl}{Infl}
\dmo{\Ker}{Ker}
\dmo{\Komp}{K}
\dmo{\Kos}{Kos}
\dmo{\modname}{\mathsf{mod}}%
\dmo{\perm}{\mathsf{perm}}
\dmo{\proj}{proj}
\dmo{\Res}{Res}
\dmo{\rmG}{G}
\dmo{\rmH}{H}
\dmo{\rmK}{K}
\dmo{\sgn}{sgn}
\dmo{\smallb}{b}
\dmo{\smallperf}{perf}
\dmo{\Spc}{Spc}
\dmo{\Spech}{Spec^h}
\dmo{\thick}{thick}

\nc{\AcatGk}{\Acat{G}{\kk}}
\nc{\AcatGR}{\Acat{G}{\CR}}
\nc{\Acat}[2]{\Aname(#1;#2)}
\nc{\adhoc}{{\sl ad hoc}\xspace}
\nc{\adj}{\dashv}
\nc{\AGk}{\KA{G}{\kk}}
\nc{\aka}{{a.\,k.\,a.}\ }
\nc{\Aname}{\cat{P}}
\nc{\BcatGk}{\Bcat{G}{\kk}}
\nc{\BcatGR}{\Bcat{G}{\CR}}
\nc{\Bcat}[2]{\Bname(#1;#2)}
\nc{\BGk}{\KB{G}{\kk}}
\nc{\Bname}{\cat{Q}}
\nc{\cat}[1]{\mathscr{#1}}
\nc{\cA}{\cat{A}}
\nc{\Cb}{\Chain_{\smallb}}
\nc{\cf}{{\sl cf.}\ }
\nc{\CRG}{\CR G}
\nc{\CR}{R}
\nc{\cT}{\cat{T}}
\nc{\Db}{\Der_{\smallb}}
\nc{\Dperf}{\Der_{\smallperf}}
\nc{\DRperf}[1][]{\Der_{\kern-0.1em\ifblank{#1}{\CR}{#1}\text{-}\kern-0.1em\smallperf}}
\nc{\eg}{{\sl e.g.}\@\xspace}
\nc{\eps}{\epsilon}
\nc{\Homcat}[1]{\Hom_{\cat #1}}
\nc{\hook}{\hookrightarrow}
\nc{\ie}{{\sl i.e.}\@\xspace}
\nc{\inv}{^{-1}}
\nc{\isoto}{\overset{\sim}{\,\to\,}}
\nc{\KA}[2]{\rmK_0^{\Aname}(#1;#2)}
\nc{\Kbac}{\Komp_{\smallb,\mathrm{ac}}}
\nc{\KB}[2]{\rmK_0^{\Bname}(#1;#2)}
\nc{\Kb}{\Komp_{\smallb}}
\nc{\KG}[2]{\rmG_0(#2#1)}
\nc{\kkG}{\kk G}
\nc{\kk}{k}
\nc{\KP}[2]{\rmK_0(#2#1)}
\nc{\Lotimes}{\otimes^{\rmL}}
\nc{\lto}{\leftarrow}
\nc{\Mid}{\,\big|\,}
\nc{\mmod}[1]{\modname(#1)}%
\nc{\normal}{\vartriangleleft}
\nc{\onto}{\mathop{\twoheadrightarrow}}
\nc{\PGk}{\KP{G}{\kk}}
\nc{\potimes}[1]{^{\otimes #1}}
\nc{\ppermutation}{$\natural$-permutation\xspace}
\nc{\pperm}{p\textrm{-}\!\perm}
\nc{\qquadtext}[1]{\qquad\textrm{#1}\qquad}
\nc{\quadtext}[1]{\quad\textrm{#1}\quad}
\nc{\RGk}{\KG{G}{\kk}}
\nc{\rmL}{\mathsf{L}}
\nc{\sbull}{{\scriptscriptstyle\bullet}}
\nc{\SET}[2]{\big\{\,#1\Mid#2\,\big\}}
\nc{\smat}[1]{\left(\begin{smallmatrix} #1 \end{smallmatrix}\right)}
\nc{\sminus}{\smallsetminus}
\nc{\too}{\mathop{\longrightarrow}\limits}
\nc{\tors}{\text{-}\mathrm{tors}}
\nc{\To}{\Rightarrow}
\nc{\xfrom}[1]{\xleftarrow{#1}}
\nc{\xto}[1]{\xrightarrow{#1}}

\usepackage{mathtools}
\usepackage{xspace}
\makeatletter
\let\ea\expandafter
\newcount\foreachcount

\def\foreachLetter#1#2#3{\foreachcount=#1
  \ea\loop\ea\ea\ea#3\@Alph\foreachcount
  \advance\foreachcount by 1
  \ifnum\foreachcount<#2\repeat}

\def\definebb#1{\ea\gdef\csname bb#1\endcsname{\ensuremath{\mathbb{#1}}\xspace}}
\foreachLetter{1}{27}{\definebb}

\makeatother

\date{\today}
\author{Paul Balmer}
\address{Paul Balmer, UCLA Mathematics Department, Los Angeles, CA 90095, USA}
\email{balmer@math.ucla.edu}
\urladdr{https://www.math.ucla.edu/~balmer}

\author{Martin Gallauer}
\address{Martin Gallauer, Max-Planck-Institut f\"ur Mathematik, 53111 Bonn, Germany}
\email{gallauer@mpim-bonn.mpg.de}
\urladdr{https://guests.mpim-bonn.mpg.de/gallauer}

\nc{\newton}{The authors would like to thank the Isaac Newton Institute for Mathematical Sciences for support and hospitality during the programme \textit{K-theory, algebraic cycles and motivic homotopy theory} when work on this paper was undertaken. This programme was supported by EPSRC grant number EP/R014604/1.}

\hypersetup{pdfauthor={Paul Balmer and Martin Gallauer},pdftitle={Finite permutation resolutions}}

\begin{document}


\title{Finite permutation resolutions}

\begin{abstract}
We prove that every finite dimensional representation of a finite group over a field of characteristic~$p$ admits a finite resolution by $p$-permutation modules.
The proof involves a reformulation in terms of derived categories.
\end{abstract}

\subjclass[2010]{}

\thanks{First-named author supported by NSF grant~DMS-1901696. Second-named author supported by a Titchmarsh Fellowship of the University of Oxford.
\newton{}}

\maketitle

\vskip-\baselineskip\vskip-\baselineskip
\tableofcontents
\vskip-\baselineskip\vskip-\baselineskip\vskip-\baselineskip

\section{Introduction}
\label{sec:intro}%


\emph{Throughout this paper, $G$ is a finite group and $\kk$ is a field of characteristic~$p>0$, typically dividing the order of~$G$. All modules are assumed finitely generated.}

\medbreak

Permutation modules are those obtained by linearizing finite $G$-sets, see \Cref{Rec:permutation}.
Letting the group vary, the class of permutation modules is also the smallest one that contains free modules and that is closed under induction and restriction along arbitrary homomorphisms.
Any permutation $\kkG$-module is isomorphic to $\kk(G/H_1)\oplus\cdots\oplus\kk(G/H_r)$ for some subgroups~$H_1,\ldots,H_r$ of~$G$, of which there are of course finitely many.
Direct summands of permutation $\kkG$-modules are called \emph{$p$-permutation} or \emph{trivial source modules}.
Despite their apparent simplicity, permutation modules play an important role in many areas of group representation theory, as recalled for instance in the introduction of Benson-Carlson~\cite{benson-carlson:cplx-permutations}. Our own interest stems from the theory of Artin motives, see Voevodsky~\cite[\S\,3.4]{Voevodsky00}.

Recall that projective resolutions of non-projective $\kkG$-modules are necessarily unbounded. Our \Cref{Thm:kG-pperm-resolutions} shows that $p$-permutation modules are significantly more flexible than projectives, in that they allow finite resolutions of all modules.

\begin{Thm}
\label{Thm:resol-intro}%
Every $\kkG$-module~$M$ admits a finite resolution by $p$-permuta\-tion $\kkG$-modules. In particular, for $G$ a $p$-group, every module admits a finite resolution by permutation modules.
\end{Thm}

It is surprising that this result could be new, in such a mature part of mathematics. A possible explanation is that it runs contrary to intuition. Bouc-Stancu-Webb~\cite{bouc-stancu-webb:projdim-mackey} show that if bounded $p$-permutation resolutions exist and moreover remain acyclic on $H$-fixed points for all subgroups~$H$ of~$G$ then the $p$-Sylow subgroups of~$G$ are very special: either cyclic or dihedral.
In broader terms, modular representation theory is well-known to be \emph{wild} for most groups, whereas permutation modules, with their finitely many isomorphism types of indecomposables, may not seem `wild enough' to control all $\kkG$-modules. In any case, conventional wisdom was that such a result would probably not hold for all finite groups.

Things changed with the partial result of~\cite{balmer-benson:permutation-resolutions}, a weaker form of \Cref{Thm:resol-intro} `up to direct summands', resolving $M\oplus N$ instead of~$M$ for some \adhoc complement~$N$. No control on~$N$ was given in~\cite{balmer-benson:permutation-resolutions}. Our first proof of \Cref{Thm:resol-intro} followed \cite{balmer-benson:permutation-resolutions} and was based on a reduction to elementary abelian groups via Carlson~\cite{carlson:induction-abelem}; the latter in turn relies on Serre's theorem on products of Bockstein~\cite{serre:bockstein}. We present here a simpler proof, which is more self-contained. We do not need to reduce to elementary abelian groups and do not invoke~\cite{serre:bockstein} at all.

\Cref{Thm:resol-intro} is an existence result, not a recipe to construct $p$-permutation resolutions. The proof does give a method to find them but it is convoluted. We leave the problem of finding effective constructions to interested readers.

\medbreak

The overarching theme we explore in this paper and the sequel~\cite{balmer-gallauer:resol-big} is how representations are controlled by permutation ones, even with more general coefficients. So \emph{let $\CR$ be a commutative noetherian ring}.
Consider the inclusion of the additive category of permutation $\CRG$-modules (\Cref{Not:permutation})
\[
\perm(G;\CR)\subseteq\mmod{\CRG}
\]
inside the abelian category of finitely generated $\CRG$-modules. This inclusion induces a canonical functor $\Upsilon\colon\Kb(\perm(G;\CR))\to \Db(\CRG)=\Db(\mmod{\CRG})$ from the bounded homotopy category of the former to the bounded derived category of the latter. The kernel of~$\Upsilon$ is the thick subcategory $\Kbac(\perm(G;\CR))$ of \emph{acyclic} complexes of permutation modules, studied in~\cite{benson-carlson:cplx-permutations}. The functor $\Upsilon$ descends to the corresponding Verdier quotient and, after idempotent-completion $(-)^\natural$, yields
\begin{equation}
\label{eq:Kb->Db}%
\bar{\Upsilon} \colon\left(\frac{\Kb(\perm(G;\CR))}{\Kbac(\perm(G;\CR))}\right)^\natural\too\Db(\CRG).
\end{equation}
This canonical functor~$\bar{\Upsilon}$ is our central object of study -- hence the eye-catching notation~$\Upsilon$.
The only formal property that $\bar{\Upsilon}$ inherits by construction is being conservative. So the first surprise is:
\begin{Thm}
\label{Thm:ff-intro}%
The canonical functor $\bar{\Upsilon}$ of~\eqref{eq:Kb->Db} is always fully faithful.
\end{Thm}

See \Cref{Thm:Kperm->A(G)}. This result employs a notion of `good' ($p$-)permutation resolution of complexes, first introduced for modules in~\cite{balmer-benson:permutation-resolutions}.
They are resolutions admitting sufficiently many projectives in low homological degrees.
We discuss this more precisely in \Cref{sec:A(G)}.
An important property of `good' resolutions, as opposed to na\"{i}ve resolutions, is that the class of complexes which admit such `good' resolutions forms a \emph{triangulated} subcategory of $\Db(\CRG)$, denoted here
\[
\BcatGR.
\]

A key fact is that the trivial $\CRG$-module $\CR$ belongs to this subcategory $\BcatGR$ when $G$ is a $p$-group.
This occupies \Cref{sec:1-resol}.
For odd primes~$p$ we describe the required `good' resolution of~$\CR$ as an explicit Koszul complex.
However, for~$p=2$ (and when $2\neq 0$ in~$\CR$), the terms in the Koszul complex need not be actual permutation modules because \emph{signs} come in the way.
We solve that issue via an induction on the order of the $2$-group and a few tricks of technical nature.

In \Cref{sec:im(F)}, we complete the proof of \Cref{Thm:ff-intro} and identify the essential image of~$\bar{\Upsilon}$ as the idempotent-completion of the aforementioned subcategory~$\BcatGR$ of~$\Db(\CRG)$. This is \Cref{Thm:Kperm->A(G)}, which holds for any~$\CR$, not just for fields.

Of course, \eqref{eq:Kb->Db} cannot be essentially surjective in general, even for $G$ trivial, unless $\CR$ is \emph{regular}. Indeed, for $G=1$ the functor~$\bar{\Upsilon}$ boils down to the canonical inclusion $\Dperf(\CR)\hook \Db(\CR)$.
Regularity of~$\CR$ turns out to be the only obstruction. We recover in this way an unpublished result of Rouquier~\cite[\S\,2.4]{beilinson-vologodsky:dg-voevodsky-motives}\,(\footnote{\,See also~\url{https://www.math.ucla.edu/~rouquier/papers/perm.pdf}}):
\begin{Thm}[\Cref{Sch:regular}]
\label{Thm:regular-intro}%
Let $\CR$ be a regular noetherian commutative ring. The canonical functor~$\bar{\Upsilon}$ of~\eqref{eq:Kb->Db} is a triangulated equivalence
\[
\left(\frac{\Kb(\perm(G;\CR))}{\Kbac(\perm(G;\CR))}\right)^\natural\xto{\ \sim\ }\Db(\CRG).
\]
\end{Thm}
This theorem will also be an easy consequence of the companion paper~\cite{balmer-gallauer:resol-big}, where we construct on~$\Db(\CRG)$ an invariant that completely characterizes the complexes in the essential image of $\bar{\Upsilon}$, for possibly singular rings~$\CR$.

\medbreak

We now come full circle and return to coefficients $\CR=\kk$ in a field~$\kk$ of characteristic~$p>0$.
Write $\pperm(G;\kk):=\perm(G;\kk)^\natural$ for the category of $p$-permutation modules.
In that situation, we have:
\begin{Thm}[\Cref{Thm:kG-pperm-resolutions}]
\label{Thm:Kperm(kG)-intro}%
The canonical functor is an equivalence
\[
\frac{\Kb(\pperm(G;\kk))}{\Kbac(\pperm(G;\kk))} \xto{\ \sim\ }\Db(\kkG).
\]
\end{Thm}

Note that this is sharper than \Cref{Thm:regular-intro} for we do not idempotent-complete the quotient. (The confused reader might want to consult \Cref{Rem:idempotent-completion}.)
The crux of the matter is the following.
\Cref{Thm:regular-intro} tells us that our category $\BcatGk$ of complexes admitting `good' $p$-permutation resolutions is a \emph{dense} triangulated subcategory of the derived category, \ie $\BcatGk^\natural=\Db(\kkG)$.
\Cref{Thm:Kperm(kG)-intro} relies on the fact that we already have $\BcatGk=\Db(\kkG)$.
We use Thomason's classification of dense subcategories in triangulated categories to reduce the proof of $\BcatGk=\Db(\kkG)$ to an equality of Grothendieck groups: $\rmK_0(\BcatGk)=\RGk$. This is accomplished in \Cref{sec:dense-and-K_0}, using Brauer's Induction Theorem in the modular case, and it completes the proof of \Cref{Thm:Kperm(kG)-intro}. Finally, for a $\kkG$-module~$M$, the information that $M$ belongs to~$\BcatGk$ says more than just $M$ being the homology in degree zero of a complex of $p$-permutation modules. It does say that $M$ has a $p$-permutation resolution. This fact is another advantage of `good' resolutions (\Cref{Cor:0-resolution}) and we obtain \Cref{Thm:resol-intro} as a consequence.

\medbreak

Let us record an easy application to tensor-triangular geometry:
\begin{Cor}
The homogeneous spectrum $\Spech(\rmH^\sbull(G,\kk))$ is an open subspace of the tt-spectrum of the tensor-triangulated category~$\Kb(\perm(G;\kk))$.
\end{Cor}

We shall return to the analysis of~$\Spc(\Kb(\perm(G;\kk)))$ and of the closed complement of $\Spech(\rmH^\sbull(G,\kk))$ in forthcoming work on Artin motives.

\begin{Ack}
We thank Robert Boltje and Serge Bouc for removing our assumption that the field~$\kk$ should be `sufficiently large' in \Cref{Prop:Bouc-Boltje}. We thank Rapha\"el Rouquier for making his \href{https://www.math.ucla.edu/~rouquier/papers/perm.pdf}{letter to Beilinson and Vologodsky} available online, in response to an earlier version of this work.
\end{Ack}


\subsection*{Notation and convention}

\

We write $\simeq$ for isomorphisms and reserve $\cong$ for canonical isomorphisms.

A commutative noetherian ring~$\CR$ is \emph{regular} if it is locally of finite projective dimension. Most results reduce to the case where $\CR$ is connected.

For a left-noetherian ring $\Lambda$, not necessarily commutative, we write $\mmod{\Lambda}$ for the category of finitely generated left $\Lambda$-modules.

We use homological notation for complexes $\cdots \to M_n\to M_{n-1}\to \cdots$. We write $\Cb$ for categories of bounded complexes, $\Kb$ for homotopy categories of bounded complexes, and $\Db$ for bounded derived categories.
We abbreviate $\Db(\Lambda)$ for $\Db(\mmod{\Lambda})$. When we speak of a module $M$ as a complex, we mean it concentrated in degree zero.

All triangulated subcategories are implicitly assumed to be replete. We abbreviate `thick' for `triangulated and thick' (\ie closed under direct summands). We write $\thick(\cA)$ for the smallest thick subcategory containing~$\cA$.

We denote by $\cA^\natural$ the idempotent-completion (Karoubi envelope) of an additive category~$\cA$, or its obvious realization in an ambient idempotent-complete category. See details in~\cite{balmer-schlichting:idempotent-completion}, including $\Kb(\cA^\natural)\cong\Kb(\cA)^\natural$.

\begin{Rem}
\label{Rem:idempotent-completion}%
The reader should distinguish the following three categories
\[
\frac{\Kb(\perm(G;\CR))}{\Kbac(\perm(G;\CR))}
\subseteq
\frac{\Kb(\perm(G;\CR)^\natural)}{\Kbac(\perm(G;\CR)^\natural)}
\subseteq
\bigg(\frac{\Kb(\perm(G;\CR))}{\Kbac(\perm(G;\CR))}\bigg)^\natural.
\]
The one on the right is the idempotent-completion of the other two.
Moreover, if $\CR=\kk$ is a field then \Cref{Thm:Kperm(kG)-intro} implies that the middle one is already idempotent-complete so that the middle one and the one on the right coincide.

This subtlety is an important point to appreciate the present work.
In general, the Verdier quotient of an idempotent-complete category does not necessarily remain idempotent-complete.
In algebraic geometry, it was Thomason's major insight in~\cite{thomason-trobaugh:pour-etre-vrai} that the only difference between the derived category~$\Dperf(U)$  of an open subscheme $U\subseteq X$ and the obvious Verdier quotient of $\Dperf(X)$ was precisely an idempotent-completion. We return to Thomason's ideas in \Cref{sec:dense-and-K_0}.
\end{Rem}

We already used the following notation. We spell it out for clarity.
\begin{Not}
\label{Not:permutation}%
For a (finite) left $G$-set~$A$ we write $\CR(A)$ for the free $\CR$-module with $G$-action extended $\CR$-linearly from its basis~$A$. An $\CRG$-module is called a \emph{permutation module} if it is isomorphic to $\CR(A)$ for some $G$-set~$A$. We denote by $\perm(G;\CR)\subseteq \mmod{\CRG}$ the subcategory of permutation $\CRG$-modules. We use the phrase `$P$ is \emph{\ppermutation}' to say that $P$ belongs to~$\perm(G;\CR)^\natural$, that is, there exists $Q$ such that $P\oplus Q$ is permutation. This is meant to evoke the following:
\end{Not}

\begin{Rec}
\label{Rec:permutation}%
If $\CR=\kk$ is a field of characteristic~$p>0$, our \ppermutation\ modules are usually called \emph{$p$-permutation modules}. They are characterized as those modules which restrict to permutation modules on $p$-Sylow subgroups, \ie they are \emph{trivial source modules}. (In particular, if $G$ is a $p$-group then $\perm(G;\kk)^\natural=\perm(G;\kk)$; see~\cite[Theorem~8]{green:indecomposables}.) 
This characterization is specific to fields of characteristic~$p$ whereas the idempotent-completion $\perm(G;\CR)^\natural$ makes sense for any ring~$\CR$.
\end{Rec}

\begin{Rem}
\label{Rem:tensor}%
We tensor $\CRG$-modules over~$\CR$ and let $G$ act diagonally. Note that the tensor of permutation modules remains permutation. Recall also that if $P$ is free and $M$ is such that $\Res^G_1 M$ is $\CR$-free then $P\otimes M$ is free, by Frobenius reciprocity.
\end{Rem}


\section{Bounded permutation resolutions}
\label{sec:A(G)}%


\emph{In this section, all complexes are bounded except if explicitly stated otherwise.} As in~\cite{balmer-benson:permutation-resolutions} we begin with a stronger notion of resolution.

\begin{Def}
\label{Def:perm-resol}%
Let $X$ be a bounded complex of $\CRG$-modules and $m\in\bbZ$. An \emph{$m$-free permutation resolution of~$X$} is a quasi-isomorphism of complexes $s\colon P\to X$ where
$P$ is a bounded complex of permutation $\CRG$-modules which is $m$-free, meaning that $P_i$ is free for all $i\le m$. Clearly $m'$-free implies $m$-free when $m'\ge m$.

Similarly (\Cref{Rec:permutation}) an \emph{$m$-projective \ppermutation\ resolution} is a quasi-isomorphism $P\to X$ where all~$P_i$ are \ppermutation, and projective for $i\le m$.
\end{Def}

\begin{Rem}
\label{Rem:variation}%
The statements of \Cref{Lem:beware}, \Cref{Cor:0-resolution} and \Cref{Prop:perm-resol} hold with the words `permutation' replaced by `\ppermutation' and with `free' replaced by `projective'. We leave most such proofs to the reader.
\end{Rem}

\begin{Rem}
\label{Rem:new-beware}%
The word `resolution' in \Cref{Def:perm-resol} can be misleading for the complex $P$ is allowed to extend further to the right than~$X$ itself, even for $X=M[0]$, a single $\CRG$-module in degree zero. This can be corrected, when $m$ is large enough:
\end{Rem}

\begin{Lem}
\label{Lem:beware}%
Let $m\ge n$ be such that the complex $X$ is acyclic strictly below degree~$n$
and such that $X$ admits an $m$-free permutation resolution. Then $X$ admits an $m$-free resolution $P\to X$ where moreover $P_i=0$ for all $i<n$.
(See \Cref{Rem:variation}.)
\end{Lem}

\begin{proof}
We can assume~$n=0$. So $m\ge 0$. Let $s\colon Q\to X$ be an $m$-free permutation (resp.\ $m$-projective \ppermutation) resolution. Since $s$ is a quasi-isomorphism, $\rmH_i(Q)=0$ for $i<0$. Since $Q_i$ is projective for $i<0$, the complex~$Q$ `splits' in negative degrees. So we have a decomposition $Q=Q'\oplus Q''$ where the subcomplex $Q'=\cdots \to Q_1\to Q_0'\to 0\to \cdots$ is concentrated in non-negative degrees whereas $Q''=\cdots \to 0\to Q_0''\to Q_{-1}\to \cdots$ lives in non-positive degrees and is acyclic, \ie the composite $Q'\to Q\to X$ remains a quasi-isomorphism. At this stage $Q_0'$ is only projective but not necessarily free. (In the case of $m$-projective \ppermutation\ resolutions, the proof stops here with $P=Q'$.) Since $Q''$ is split exact, we see that $Q_0''$ is stably free: $Q_0''\oplus(\oplus_{i<0\textrm{,even}}Q_i)\simeq \oplus_{i<0\textrm{,odd}}Q_i$. Since $Q_0'\oplus Q_0''=Q_0$ is free, we see that $Q_0'$ is also stably free, namely $Q_0'\oplus L$ is free for the free~$L=\oplus_{i<0\textrm{,odd}}Q_i$. Adding to $Q'$ the complex $0\to L\xto{1} L\to 0$ with $L$ in degrees~1 and 0 (with zero map to~$X$), we obtain a new complex~$P$ and a quasi-isomorphism $P\to X$, where now $P_0=Q_0'\oplus L$ is free and $P_1=Q_1\oplus L$ is permutation (and free if~$Q_1$ was) and $P_i=Q_i$ for $i>1$ and $P_i=0$ for $i<0$. This $P\to X$ is the desired resolution.
\end{proof}

\begin{Cor}
\label{Cor:0-resolution}%
Let $M\in\mmod{\CRG}$ be such that, when viewed as a complex, $M$ admits a $0$-free permutation resolution in the sense of \Cref{Def:perm-resol}. Then $M$ admits a finite resolution by finitely generated permutation modules. (See \Cref{Rem:variation}.)
\qed
\end{Cor}

\begin{Lem}
\label{Lem:lift}%
Let $P,X,Y$ be bounded complexes, let $f\colon P\to X$ and $s\colon Y\to X$ be morphisms in~$\Cb(\CRG)$ such that $s$ is a quasi-isomorphism:
\[
\xymatrix@R=.1em{
& Y\ar[dd]^-{s}
\\
P \ar[rd]_-{f} \ar@{..>}[ru]^-{\hat f}
\\
& X.\!}
\]
Suppose that $m\in\bbZ$ is such that $X_i=0$ and $Y_i=0$ for $i>m$ and $P_i$ is projective for all~$i\le m$. Then there exists $\hat f\colon P\to Y$ such that $s\,\hat f$ is homotopic to~$f$.
\end{Lem}

\begin{proof}
Let $Z$ be an acyclic complex such that $Z_i=0$ for $i>m+1$. Then any map $P\to Z$ is null-homotopic, as one can build a homotopy using the usual induction argument, that only requires $P_i$ projective for $i\le m$ to lift against the epimorphism $Z_{i+1}\onto \im(Z_{i+1}\to Z_{i})$. Now $Z:=\cone(s)$ is such a complex. So the composite $P\xto{f} X\to \cone(s)$ is zero in~$\Kb(\mmod{\CRG})$. Hence $f$ factors as claimed.
\end{proof}

\begin{Prop}
\label{Prop:perm-resol}%
Let $s\colon Y\to X$ be a quasi-isomorphism. Then $X$ admits $m$-free permutation resolutions for all $m\ge 0$ if and only if~$Y$ does. (See \Cref{Rem:variation}.)
\end{Prop}

\begin{proof}
From~$Y$ to~$X$ is trivial. So suppose that $X$ has the property and let us show it for~$Y$. Let $m\ge 0$. Increasing~$m$ if necessary, we can assume that $X_i=Y_i=0$ for all~$i>m$. Let then $f\colon P\to X$ be an $m$-free permutation resolution of~$X$. By \Cref{Lem:lift}, there exists $\hat f\colon P\to Y$ such that $s\, \hat f\sim f$. By 2-out-of-3, $\hat f\colon P\to Y$ is a quasi-isomorphism, hence yields an $m$-free permutation resolution of~$Y$.
\end{proof}

So far we dealt with complexes on the nose, in $\Cb(\CRG)$. The above proposition allows us to pass to the derived category, if we make sure to require the existence of $m$-free permutation resolutions for \emph{all}~$m\ge0$.
\begin{Def}
\label{Def:A(G)}%
Using the special permutation resolutions of \Cref{Def:perm-resol}, we have two well-defined replete subcategories of the derived category
\begin{align*}
\AcatGR&=\left\{ X\in \Db(\CRG)\ \bigg|\ {X\textrm{ admits $m$-free permutation} \atop \textrm{resolutions, for all }m\ge 0}\ \right\}\quadtext{and}
\\
\BcatGR&=\left\{ X\in \Db(\CRG)\ \bigg|\ {X\textrm{ admits $m$-projective \ppermutation} \atop \textrm{resolutions, for all }m\ge 0}\ \right\}.
\end{align*}
\end{Def}

\begin{Prop}
\label{Prop:A(G)-B(G)}%
The two subcategories $\AcatGR\subseteq\BcatGR$ above are triangulated subcategories of~$\Db(\CRG)$.
\end{Prop}

\begin{proof}
We prove it for $\AcatGR$. The proof for $\BcatGR$ is similar. It suffices to show that if $f\colon X\to Y$ is a morphism in~$\Db(\CRG)$ with $X,Y\in\AcatGR$ then $\cone(f)\in\AcatGR$. The morphism $f$ is represented by a fraction $X\overset{s}\lto Z\xto{g}Y$ in~$\Cb(\CRG)$ with $s$ a quasi-isomorphism. By \Cref{Prop:perm-resol}, we have $Z\in \AcatGR$. Since $\cone(f)\simeq \cone(g)$ in~$\Db(\CRG)$, we can assume that $f\colon X\to Y$ is a plain morphism of complexes. Let now $n\ge 0$. Since $Y$ belongs to~$\AcatGR$, choose an $n$-free permutation resolution $t\colon Q\to Y$. Choose $m\gg n$ such that $Q_i=Y_i=0$ for $i>m$, using that $Q$ and~$Y$ are bounded. Since $X\in\AcatGR$, choose an $m$-free permutation resolution $s\colon P\to X$. We thus have the (plain) morphisms~$f,s,t$:
\[
\xymatrix@R=1.5em{
P\ar[d]_-{s} \ar@{..>}[r]_-{h}
& Q \ar[d]^-{t}
\\
X\ar[r]^-{f}
& Y
}
\]
By \Cref{Lem:lift}, there exists $h\colon P\to Q$ such that $t\,h\sim f\, s$. Since $s$ and~$t$ are quasi-isomorphisms, so is the induced map $\cone(h)\to \cone(f)$. Now, the mapping cone of~$h$ is a complex of permutation modules that is free in degrees~$\le n$ since $P$ and~$Q$ are. As $n\ge0$ was arbitrary, we proved $\cone(f)\in\AcatGR$ as claimed.
\end{proof}

\begin{Rem}
\label{Rem:p-gp}%
Let $G$ be a $p$-group and $\CR=\kk$ a field of characteristic~$p$. Then $\Db(\kkG)$ is generated as a triangulated subcategory by~$\kk$. So, by \Cref{Prop:A(G)-B(G)}, the triangulated subcategory $\AcatGk=\BcatGk$ is equal to $\Db(\kkG)$ if and only if it contains~$\kk$.
In fact, we will prove in \Cref{Thm:kG-pperm-resolutions} that $\BcatGk=\Db(\kkG)$ holds, for all finite groups.
\end{Rem}

\begin{Rem}
\label{Rem:M-A(G)}%
It follows easily from \Cref{Prop:A(G)-B(G)} that an $\CRG$-module~$M$ belongs to~$\BcatGR$ if and only if all its Heller loops (syzygies in a projective resolution) admit finite \ppermutation resolutions. This can be sharpened as follows.
\end{Rem}

\begin{Prop}
\label{Prop:M-resol-lim}%
Let $M$ be an $\CRG$-module such that $M$ belongs to~$\BcatGR$.
Let $\cdots \to P_n\to \cdots \to P_0\xto{\pi} M\to 0$ be a possibly unbounded resolution of~$M$ by finitely generated projective $\CRG$-modules, viewed as a quasi-isomorphism~$P\to M$. Then there exists a sequence of complexes~$\{Q(n)\}_{n\in\bbN}$ in $\Chain_{\geq 0}(\CRG)$ and a commutative diagram of quasi-isomorphisms
\[
\xymatrix{
P \ar[d]_-{\pi} \ar[rd] \ar@{}[rrd]|-{\cdots} \ar@/^.5em/[rrrd] \ar@/^1em/[rrrrd] \ar@{}[rrrrrd]^-{\kern10em\cdots}
\\
\kern3em M=Q(0) 
& Q(1) \ar[l] 
& \cdots \ar[l]
& Q(n) \ar[l] 
& Q(n+1) \ar[l] 
& \cdots \ar[l]
}
\]
such that for each $n\geq 1$:
\begin{enumerate}[\rm(1)]
\item
\label{it:M-resol-lim-i}%
The complex~$Q(n)$ is bounded and consists of \ppermutation $\CRG$-modules.
\item
\label{it:M-resol-lim-ii}%
The map $P_d\to Q(n)_d$ is the identity for all~$d<n$.
\end{enumerate}
In particular, the sequence~$\cdots \to Q(n)\to \cdots \to Q(0)$ eventually stabilizes in each degree and $P$ is the limit of that sequence in~$\Chain(\CRG)$.
\end{Prop}

\begin{proof}
Suppose we have the factorization via~$Q(n)$ for some~$n\ge 0$ satisfying~\eqref{it:M-resol-lim-i} and~\eqref{it:M-resol-lim-ii} if $n\geq 1$. Since $Q(n)$ is bounded, there exists $m\ge n$ such that the quasi-isomorphism $P\to Q(n)$ factors as $P\to P'\to Q(n)$ via the canonical $m$-truncation~$P'$ of~$P$
\[
\xymatrix@R=1em{
\kern1em P= \ar[d]
&& \cdots \ar[r]
& P_{m+2} \ar[r] \ar[d]
& P_{m+1} \ar[r] \ar@{->>}[d]
& P_{m} \ar[r] \ar@{=}[d]
& P_{m-1} \ar[r] \ar@{=}[d]
& \cdots
\\
\kern1em P':=
&& \cdots \ar[r]
& 0 \ar[r]
& N \ar[r]
& P_{m} \ar[r]
& P_{m-1} \ar[r]
& \cdots
}
\]
where $N=\im(P_{m+1}\to P_m)$. By \Cref{Prop:A(G)-B(G)}, we know that $N$ still belongs to~$\BcatGR$ and in particular it admits a finite \ppermutation resolution. Let $Q(n+1)$ be the complex obtained by replacing $N$ in~$P'$ by this \ppermutation resolution, via splicing. So we have the four (plain) quasi-isomorphisms
\[
\xymatrix@R=1em{
P \ar[d] \ar[rd] \ar@/^.5em/@{-->}[rrd]^-{\exists}
\\
Q(n)
& P' \ar[l]
& Q(n+1). \ar[l]
}
\]
There exists a dashed arrow $P\to Q(n+1)$ making the diagram commute because $P$ is a complex of projectives. By construction, $Q(n+1)$ satisfies~\eqref{it:M-resol-lim-i}, as well as~\eqref{it:M-resol-lim-ii} since $m\ge n$ and since the maps $Q(n+1)_d\to P'_d\lto P_d$ are the identity for all~$d\le m$. The map $Q(n+1)\to Q(n)$ is the above composite.
\end{proof}

\begin{Rec}
\label{Rec:dense}%
A triangulated subcategory $\cA\subseteq\cT$ is \emph{dense} if every object~$X$ of~$\cT$ is a direct summand of an object $X\oplus Y$ of~$\cA$. This amounts to $X\oplus \Sigma X\in\cA$ since $X\oplus \Sigma X=\cone\big(\smat{0&0\\0&1}\colon X\oplus Y\to X\oplus Y\big)$. Hence the thick closure of~$\cA$ is
\[
\thick(\cA)=\SET{X\in\cT}{\exists\,Y\in\cT\textrm{ s.t.\ }X\oplus Y\in\cA}=\SET{X\in\cT}{X\oplus \Sigma X\in\cA}.
\]
When $\cT$ is idempotent-complete (like~$\Db(\CR G)$ here) we have $\thick(\cA)=\cA^\natural$.
\end{Rec}

\begin{Prop}
\label{Prop:A-dense-in-B}%
The triangulated subcategory $\AcatGR\subseteq\BcatGR$ is dense.
\end{Prop}

\begin{proof}
It suffices to show that for every $X\in\BcatGR$, we have $X\oplus\Sigma X\in\AcatGR$. By \Cref{Def:perm-resol}, it suffices to show that if $P$ is an $m$-projective complex of \ppermutation\ $\CRG$-modules for some~$m\ge 0$ then $P\oplus \Sigma P$ is homotopy equivalent to an $m$-free complex~$\tilde{P}$ of permutation $\CRG$-modules. For each~$i$, since $P_i$ is \ppermutation\ (resp.\ projective for $i\le m$), there exists $Q_i$ such that $P_i\oplus Q_i$ is permutation (resp.\ free for $i\le m$). Adding to $P\oplus \Sigma P$ short complexes $\cdots 0\to Q_i\xto{1} Q_i\to 0\cdots$, with $Q_i$ in degrees~$i+1$ and~$i$, yields the wanted $m$-free complex~$\tilde{P}$ of permutation $\CRG$-modules.
\end{proof}

\begin{Rem}
\label{Rem:where-to}%
The thick closure of the subcategory $\AcatGR$ of~$\Db(\CRG)$
\[
\AcatGR^\natural=\BcatGR^\natural=\thick(\AcatGR)=\thick(\BcatGR)
\]
is a key player in this paper, as it will turn out (\Cref{Thm:Kperm->A(G)}) to be the essential image of the functor~$\bar{\Upsilon}$ in~\eqref{eq:Kb->Db}.
We also want to decide when $\AcatGR^\natural$ or $\AcatGR$ or $\BcatGR$ coincide with~$\Db(\CRG)$; see Sections~\ref{sec:im(F)} and~\ref{sec:dense-and-K_0}.
More ambitiously, we want an invariant on~$\Db(\CRG)$ that detects~$\AcatGR^\natural$. This is the subject of~\cite{balmer-gallauer:resol-big}.
\end{Rem}

\begin{Exa}
\label{Exa:G=1}%
For~$G$ trivial, $\Acat{1}{\CR}^\natural=\Dperf(\CR)$, the perfect complexes.
\end{Exa}

Let us start with generalities about the Mackey 2-functor~$\Acat{?}{\CR}$.

\begin{Rec}
\label{Rec:Ind-Res}%
Recall that for a subgroup $H\le G$, induction is a two-sided adjoint to restriction. In particular, these functors preserve injectives and projectives and yield well-defined functors on derived categories without need to derive them on either side. Recall also that the composite (for $Y$ a module or a complex over~$H$)
\begin{equation}
\label{eq:Ind-Res-separable}%
\xymatrix{
Y\ar[r]^-{\eta^{\ell}} \ar@{-<} `d/1em[r]`/1em[rr]|-{\ \id\ } [rr]
& \Res^G_H\Ind_H^G(Y)\ar[r]^-{\eps^r}
& Y
}
\end{equation}
of (usual) unit for the $\Ind\adj\Res$ adjunction and (usual) counit for the $\Res\adj\Ind$ adjunction is the identity. The other composite (for~$X$ over~$G$)
\begin{equation}
\label{eq:Ind-Res-cohomological}%
\xymatrix{
X\ar[r]^-{\eta^{r}} \ar@{-<} `d/1em[r]`/1em[rr]|-{\ [G:H]\cdot\ } [rr]
& \Ind_H^G\Res^G_H(X)\ar[r]^-{\eps^{\ell}}
& X
}
\end{equation}
of (usual) $\Res\adj\Ind$ unit and $\Ind\adj\Res$ counit is multiplication by~$[G:H]$.
\end{Rec}

\begin{Prop}
\label{Prop:Ind-Res}%
Let $H\le G$ be a subgroup. Let $X\in\Db(\CRG)$ and $Y\in\Db(\CR H)$.
\begin{enumerate}[\rm(a)]
\item
\label{it:Res}%
If $X$ belongs to~$\AcatGR^\natural$ then $\Res^G_H X$ belongs to~$\Acat{H}{\CR}^\natural$.
\smallbreak
\item
\label{it:Ind}%
$Y$ belongs to~$\Acat{H}{\CR}^\natural$ if and only if~$\Ind_H^G Y$ belongs to~$\AcatGR^\natural$.
\smallbreak
\item
\label{it:Res-[G:H]-inv}%
If $\Res^G_H X$ belongs to~$\Acat{H}{\CR}^\natural$ and multiplication by~$[G:H]$ is invertible on~$X$ (\eg\ if $[G:H]$ is invertible in~$\CR$) then $X$ belongs to~$\AcatGR^\natural$.
\end{enumerate}
\end{Prop}

\begin{proof}
Direct from $\Res^G_H$ and $\Ind_H^G$ being exact and preserving permutation modules and free modules, using~\eqref{eq:Ind-Res-separable} in the `if' part of~\eqref{it:Ind} and using~\eqref{eq:Ind-Res-cohomological} in~\eqref{it:Res-[G:H]-inv}.
\end{proof}

\begin{Cor}
\label{Cor:Acat-Sylow}%
The following are equivalent for $X\in\Db(\CRG)$:
\begin{enumerate}[\rm(i)]
\item
\label{it:AS.Acat}%
The complex $X$ belongs to~$\AcatGR^\natural$.
\item
\label{it:AS.Sylow}%
Its restriction $\Res_H^G X$ belongs to~$\Acat{H}{\CR}^\natural$ for every Sylow subgroup $H\le G$.
\end{enumerate}
\end{Cor}
\begin{proof}
The implication \eqref{it:AS.Acat}$\Rightarrow$\eqref{it:AS.Sylow} is immediate from \Cref{Prop:Ind-Res}\,\eqref{it:Res}.
For the converse, choose a $p$-Sylow subgroup $H_p$ of $G$ for each prime $p$ dividing the order of~$G$.
The integers $\{[G:H_p]\mid p\textrm{ divides }|G|\}$ are coprime hence there exist integers $a_p$ with $\sum_p a_p[G:H_p]=1$.
It follows that the sum of the modified composites~\eqref{eq:Ind-Res-cohomological}
\[
X\xto{(a_p\cdot\eta^r)_p}\bigoplus_{p\textrm{ divides }|G|}\Ind^G_{H_p}\Res^G_{H_p}X\xto{(\epsilon^r)_p}X
\]
is the identity. In particular, $X$ is a direct summand of the sum of induced complexes in the middle.
The claim now follows from \Cref{Prop:Ind-Res}\,\eqref{it:Ind} and the fact that $\AcatGR^\natural$ is idempotent complete.
\end{proof}
\begin{Def}
\label{Def:R-perfect}%
We say that a complex $X\in\Db(\CRG)$ is \emph{$\CR$-perfect} if the underlying complex $\Res^G_1(X)$ is perfect over~$\CR$. This defines a thick subcategory of~$\Db(\CRG)$
\[
\DRperf(\CRG):=\SET{X\in\Db(\CRG)}{\Res^G_1(X)\in\Dperf(\CR)}.
\]
\end{Def}

\begin{Cor}
\label{Cor:A(G)-R-perfect}%
We have $\AcatGR^\natural\subseteq\DRperf(\CRG)$. If the order~$|G|$ is invertible in~$\CR$, then $\AcatGR^\natural=\DRperf(\CRG)$.
\end{Cor}

\begin{proof}
Direct from \Cref{Prop:Ind-Res} for $H=1$ and from \Cref{Exa:G=1}.
\end{proof}

In view of \Cref{Cor:A(G)-R-perfect}, it is interesting to see what happens when we localize~$\CR$.
\begin{Lem}
\label{lift-cpx-permutation}%
Let $r\in\CR$ and set $\CR'=\CR[1/r]$. The canonical functor $\CR'\otimes_\CR-\colon$ $\Cb(\perm(G;\CR))\to\Cb(\perm(G;\CR'))$ is essentially surjective.
\end{Lem}
\begin{proof}
Pick~$P'\in\Cb(\perm(G;\CR'))$. We can assume $P_{i}=0$ unless~$0\le i\le n$. Each $P'_i=\CR'(A_i)$ for some finite $G$-set~$A_i$ clearly comes from~$\CR(A_i)$ over~$\CR$. For every integer $N\ge 1$ consider the following construction (note the changing powers of~$r$, vertically):
\[
\xymatrix@R=1.8em@C=2em{
P(N):= \ar@<-.5em>[d]^-{s(N):=}
& \cdots 0 \ar[r]
& P'_{n} \ar[r]^-{\partial_{n}} \ar[d]^-{r^{nN}}
& P'_{n-1} \ar[r]^-{\partial_{n-1}} \ar[d]^-{r^{(n-1)N}}
& \cdots \ar[r]^-{\partial_{2}}
& P'_{1} \ar[r]^-{\partial_{1}} \ar[d]^-{r^{N}}
& P'_{0} \ar[r] \ar[d]^-{1}
& 0 \cdots
\\
P'=
& \cdots 0 \ar[r]
& P'_{n} \ar[r]_-{\partial'_{n}}
& P'_{n-1} \ar[r]_-{\partial'_{n-1}}
& \cdots \ar[r]_-{\partial'_{2}}
& P'_{1} \ar[r]_-{\partial'_{1}}
& P'_{0} \ar[r]
& 0 \cdots
}
\]
where $\partial_{i}:=r^N\partial'_{i}$. This morphism is an isomorphism~$s(N)\colon P(N)\isoto P'$ in~$\Cb(\perm(G;\CR'))$. Increasing $N\gg1$, one easily arranges that the maps~$\partial_{i}$ in~$P(N)$ are defined over~$\CR$, that they are $\CRG$-linear and finally that they form a complex, for all these properties only involve finitely many denominators. Then $P(N)$ provides a source of~$P'$.
\end{proof}

\begin{Prop}
\label{Prop:inverting-p}%
Let $r\in \CR$ and set $\CR'=\CR[1/r]$. Let $X\in\Db(\CRG)$ be such that $\CR'\otimes_{\CR}X\in\Db(\CR'G)$ belongs to~$\Acat{G}{\CR'}$.
Then there exists an exact triangle $P\to X \to T\to\Sigma P$ in $\Db(\CRG)$
where $P\in\Cb(\perm(G;\CR))$ and $T\in \Db^{r\tors}(\CRG)$ is $r$-torsion, meaning that $r^n\!\cdot \id_{T}=0$ in~$\Db(\CRG)$ for $n\gg1$ large enough.
\end{Prop}

\begin{proof}
By \Cref{lift-cpx-permutation}, there exists $P\in\Cb(\perm(G;\CR))$ and an isomorphism $\CR'\otimes_{\CR}P\isoto \CR'\otimes_{\CR}X$ in~$\Db(\CR'G)$.
By the usual localization sequence
\[
\xymatrix@C=1.5em{
\Db^{r\tors}(\CRG)\ \ar@{>->}[rr]^-{\incl}
&& \Db(\CRG) \ar@{->>}[rrr]^-{\CR'\otimes_{\CR}-}
&&& \Db(\CR'G)
}
\]
(see Keller~\cite[Lemma~1.15]{keller:cyclic-exact}), we have $\Db(\CR'G)=\Db(\CRG)[1/r]$.
So the isomorphism $\CR'\otimes_{\CR}P\isoto \CR'\otimes_{\CR}X$ is given by a fraction $P\xfrom{r^n}P\xto{f}X$ in $\Db(\CRG)$ for some $n\gg1$ and some $f\colon P\to X$ such that $T:=\cone(f)$ belongs to~$\Db^{r\tors}(\CRG)$.
\end{proof}

\begin{Cor}
\label{Cor:inverting-p}%
Let $G$ be a $p$-group for some prime~$p$ and let $X\in\Db(\CRG)$ be $\CR$-perfect (\Cref{Def:R-perfect}).
Then there exists an exact triangle in $\Db(\CRG)$
\[
P\to X\oplus\Sigma X \to T\to\Sigma P
\]
where $P$ belongs to~$\Cb(\perm(G;\CR))$ and $T$ is $p$-torsion in~$\Db(\CRG)$.
\end{Cor}

\begin{proof}
We apply \Cref{Prop:inverting-p} for $r=p$, so $\CR'=\CR[1/p]$. It is easy to check that $X':=\CR'\otimes_{\CR}X\in\Db(\CR'G)$ remains $\CR'$-perfect. By \Cref{Cor:A(G)-R-perfect}, we have $X'\in\Acat{G}{\CR'}^\natural$ hence $X'\oplus\Sigma X'\in\Acat{G}{\CR'}$. We conclude by \Cref{Prop:inverting-p}.
\end{proof}

\begin{Rem}
\label{Rem:torsion}%
If $X$ is $p$-torsion, say $p^n\cdot\id_X=0$, the octahedron axiom gives
\[
X\in\thick(X\oplus\Sigma X)=\thick(\cone(X\xto{p^n}X))\subseteq\thick(\cone(X\xto{p}X)).
\]
\end{Rem}


\section{Permutation resolutions of the trivial module}
\label{sec:1-resol}%


The goal of this section is to prove:
\begin{Thm}
\label{Thm:1-resol}%
Let $\CR$ be a commutative ring and $G$ be a $p$-group for some prime~$p$. Then there exists a bounded acyclic complex of permutation~$\CRG$-modules
\begin{equation}
\label{eq:1-resol}%
C_\sbull=\cdots 0 \to C_n\to C_{n-1}\to \cdots \to C_1\to C_0 \to 0\cdots
\end{equation}
such that $C_0=\CR$ (with trivial $G$-action) and $C_1$ is a free $\CRG$-module.
\end{Thm}

We use tensor-induction; see for instance~\cite[\S\,I.3.15 and~II.4.1]{benson:representation-cohomology}.
\begin{Rec}
\label{Rec:tens-induction}%
Let $H\normal G$ be a normal subgroup of index~$n$. Choose $g_1,\ldots,g_n\in G$ a complete set of representatives of~$G/H$. Using the bijection $\{1,\ldots,n\}\isoto G/H$, $j\mapsto[g_j]$, the left $G$-action on~$G/H$ yields a group homomorphism $\sigma\colon G\to S_n$ to the symmetric group on~$n$ elements. Consider the action of~$S_n$ on~$H^n$ permuting the factors and the associated semi-direct product~$S_n\ltimes H^n$.
Define an injective group homomorphism
\[
i\colon G\hook S_n\ltimes H^n
\]
by mapping $g\in G$ to $(\sigma g,h_1,\ldots,h_n)$ where $g\cdot g_j=g_{(\sigma g)(j)}\cdot h_j$ for all~$j=1,\ldots,n$.

We write $i^*$ for the restriction of $(S_n\ltimes H^n)$-sets or $\CR(S_n\ltimes H^n)$-modules along~$i$.
The \emph{tensor-induction} of an $\CR H$-module~$N$ is the $\CRG$-module that we denote
\[
N\potimes{G/H}:=i^*(N\potimes{n});
\]
here $N\potimes{n}=N\otimes_\CR\cdots\otimes_\CR N$, with $n=[G:H]$ factors, is acted upon by~$S_n\ltimes H^n$ via permutation of factors for~$S_n$ and factorwise action for~$H^n$, that is, $(h_1,\ldots,h_n)\cdot (x_1\otimes \ldots\otimes x_n)=(h_1\,x_1)\otimes \ldots\otimes (h_n\,x_n)$.
\end{Rec}

\begin{Exa}
\label{Exa:tens-ind-free}%
Consider the $(S_n\ltimes H^n)$-set $A=\sqcup_{j=1}^n H$ on which $S_n$ acts by permuting the summands and~$H^n$ acts via $(h_1,\ldots,h_n)\cdot x=h_j\cdot x$ if~$x$ belongs to the $j$-th summand~$H$ of~$A$. Although $A$ is not free, one can check that $i^*A$ is a free $G$-set. Linearizing, let $N=\CR(A)=\oplus_{i=1}^n \CR H$ with $S_n$ permuting the summands and~$H^n$ acting via $(h_1,\ldots,h_n)\cdot(x_1,\ldots,x_n)=(h_1\,x_1,\ldots,h_n\,x_n)$. Again, $N$ is not necessarily $\CR(S_n\ltimes H^n)$-free but $i^*N$ is $\CRG$-free.

\end{Exa}

\begin{Rec}
\label{Rec:tens-ind-complexes}%
One can perform the above construction in other symmetric monoidal categories instead of $\mmod{\CR}$, since we only need $S_n$ to act on~$N\potimes{n}$. Explicitly, it suffices to know what transpositions do on~$N\potimes{n}$, and this is the swap of factors given in the symmetric monoidal structure. We can then do this for $\otimes_{\CR}$ in $\Cb(\CR)$ for instance. This is straightforward but makes signs appear.

So, with the above notation for $H\lhd G$ of index~$n$, if we take a bounded complex $C\in\Cb(\CR H)$ of $\CR H$-modules then we write
\[
C\potimes{G/H}=i^*(C\potimes{n})
\]
where the action of $S_n$ on~$C\potimes{n}$ involves signs, following the Koszul rule. This $C\potimes{G/H}$ is a bounded complex of~$\CRG$-modules, which in degree~$s$ is
\[
\bigoplus_{{r_1,\ldots,r_n}\atop{r_1+\cdots+r_n=s}}C_{r_1}\otimes_{\CR}\cdots\otimes_{\CR}C_{r_n}
\]
and whose $G$-action is restricted via $i\colon G\hook S_n\ltimes H^n$ from the `obvious' action of $S_n\ltimes H^n$ as above. For instance, a transposition $(k\ell)$ will swap the factors $C_{r_k}$ and $C_{r_\ell}$ and leave the others unchanged (typically landing in another term of the above direct sum), multiplied by a sign if both $r_k$ and~$r_\ell$ are odd.
\end{Rec}

The above construction motivates the following definition.
\begin{Def}
\label{Def:sign-perm}%
An $\CRG$-module $M$ is called \emph{sign-permutation} if it admits a \emph{sign-permutation basis}, that is,
an $\CR$-basis $A$ such that $G\cdot A\subseteq A\cup (-A)$.
\end{Def}

\begin{Exa}
\label{Exa:tens-ind-H=1}%
Consider the case $H=1$, so $n=|G|$. Then $C\potimes{G/1}$ is $C\potimes{n}$ where $G$ acts by `permuting the indices', with Koszul rule. This depends on an ordering of the elements of~$G$. Let us take a special case where \mbox{$C=(0\to \CR\xto{1}\CR\to 0)$} concentrated in degrees one and zero. Then direct inspection shows that $C\potimes{G/1}$ is the Koszul complex $\Kos(G;\CR)$ associated with the augmentation morphism $\epsilon:\CR G\to \CR$ considered as a morphism of $\CR$-modules $\CR^{n}\to\CR$ which is the identity on each summand.
Explicitly, $\Kos(G;\CR)$ is concentrated in degrees $n,\ldots,0$, and in degree $s$ is the exterior power $\Lambda^s_{\CR}(\CRG)$ with action given by \mbox{$g\cdot(v_1\wedge\cdots\wedge v_s)=$}
\mbox{$(g\,v_1)\wedge\cdots\wedge(g\,v_s)$}.
Note that $\Kos(G;\CR)_s$ has the standard $\CR$-basis given by~$A_s:=\SET{g_{i_1}\wedge\cdots\wedge g_{i_s}}{1\le i_1< \cdots<i_s\le n}$ in the chosen numbering $g_1,\ldots,g_n$ of the elements of~$G$.
The action of $S_n\ltimes H^n=S_n$ and therefore of~$G$ (via~$i$) essentially permutes this basis~$A_s$, except that it introduces signs.
The differential $\Kos(G;\CR)_s\to\Kos(G;\CR)_{s-1}$ sends a basis element $g_{i_1}\wedge\cdots\wedge g_{i_s}$ to the sum
\[
\sum_{j=1}^s(-1)^{j-1}g_{i_1}\wedge\cdots\wedge \widehat{g_{i_j}}\wedge\cdots\wedge g_{i_s}
\]
where the factor~$\widehat{g_{i_j}}$ is removed.
We conclude that $\Kos(G;\CR)$ is an acyclic complex of sign-permutation modules in the sense of \Cref{Def:sign-perm}.
Note also that $\Kos(G;\CR)_0=\CR$ and $\Kos(G;\CR)_1=\CRG$.
\end{Exa}

The proof of \Cref{Thm:1-resol} will differ depending on $p$ being even or odd.

\begin{Lem}
\label{Lem:sign-perm-odd}%
Let $p$ be an odd prime and $G$ be a $p$-group. Every sign-permutation $\CRG$-module is a permutation module.
\end{Lem}

\begin{proof}
If $1=-1$ in~$\CR$ there is nothing to prove. So assume $1\neq -1$. Let $M$ be an $\CRG$-module with sign-permutation basis $A$. Let $m=\dim_\CR(M)=|A|$, so that $M\simeq\CR^m$ as $\CR$-modules. Embed the elementary abelian 2-group $(C_2)^m\hook\Gl_m(\CR)$ as $\{\pm1\}^m$ diagonally and the symmetric group $S_m\hook\Gl_m(\CR)$ as permutation matrices and let $\Gamma=(C_2)^m\cdot S_m\le\Gl_m(\CR)$ the subgroup of matrices that have exactly one $\pm 1$ entry in each row and each column and all other entries zero. One shows easily that $\Gamma=S_m\ltimes (C_2)^m$ and in any case $[\Gamma:S_m]=2^m$. The basis~$A$ yields a group homomorphism $f\colon G\to \Gl_m(\CR)$ that lands inside~$\Gamma$ by hypothesis.

Now $f(G)$ is a $p$-subgroup of~$\Gamma$, hence is contained in a $p$-Sylow. Since $[\Gamma:S_m]=2^m$ is prime to~$p$, the image $f(G)$ is $\Gamma$-conjugate to a subgroup of~$S_m$. This means that, up to reordering and changing some signs in the basis~$A$, we can assume that $G$ acts on~$A$ via the action of~$S_m$ on~$\CR^m$, that is, by permuting the basis.
\end{proof}

\begin{Lem}
\label{Lem:sign-perm-even}%
Let $G$ be a $2$-group and $H\normal G$ be a subgroup of index~$2$. Let $L=\Infl_{G/H}^{G}(\CR_{\sgn})$ be the one-dimensional sign representation~$\CR_{\sgn}$ of the cyclic group~$G/H$ of order~$2$, on which $G$ acts via~$G\onto G/H$.
Let $M$ be an $\CRG$-module with a sign-permutation basis~$A$ (\Cref{Def:sign-perm}) and suppose that $A$ is permuted by~$H$, that is, $H\cdot A\subseteq A$ without signs.
Then there exists a decomposition of~$M$
\[
M=M^+\oplus (L\otimes M^-)
\]
as an $\CRG$-module, where $M^+$ and $M^-$ are permutation $\CRG$-modules.
\end{Lem}

\begin{proof}
Again we can assume $1\neq -1$ in~$\CR$. Pick $g\in G\sminus H$. Decompose~$A=A^+\sqcup A^-$ where $A^{+}=\SET{a\in A}{g\cdot a\in A}$ and $A^{-}=\SET{a\in A}{g\cdot a\in -A}$. Both $A^+$ and~$A^-$ are $H$-subsets of~$A$ because of the assumption on the $H$-action and because $H$ is normal in~$G$. Consider the $\CR$-submodules $M^+$ and $N$ generated inside~$M$ by $A^+$ and by~$A^-$ respectively. It is easy to see that $M^+$ is an $\CRG$-submodule. In fact, $N=\CR A^-$ is also an $\CRG$-submodule, because $g^2\in H$ gives us~$g\cdot A^-\subseteq -A^-$. Hence $M=M^+ \oplus N$ as $\CR G$-module and $M^+=\CR(A^+)$ is a permutation $\CRG$-module. The submodule $N=\CR A^-$ is not permutation because $g$, and in fact any element of~$G\sminus H$, acts on~$A^-$ by a permutation of~$A^-$ times~$-1$. Therefore $M^-:=L\otimes N$ is permutation and $M=M^+\oplus N=M^+\oplus (L\otimes M^-)$.
\end{proof}

\begin{proof}[Proof of \Cref{Thm:1-resol}]
The case of $p$ odd is direct from the sign-permutation resolution $\Kos(G;\CR)=(0\to \CR\xto{1}\CR\to 0)\potimes{G/1}$ of \Cref{Exa:tens-ind-H=1} combined with \Cref{Lem:sign-perm-odd}.

So suppose that $p=2$ and proceed by induction on the order of the $2$-group~$G$. By induction hypothesis, we can assume the result for any index-$2$ subgroup~$H\normal G$. Let $D$ be a bounded complex as in~\eqref{eq:1-resol} but for~$H$; so we have $D_0=\CR$ and $D_1$ is $\CR H$-free and all $D_j$ are permutation $\CR H$-modules. Consider now for $m\geq 0$:

{\it Claim}~$(\Delta)_m$: There exists an acyclic complex~$C$ of~$\CRG$-modules concentrated in non-negative degrees such that:
\begin{enumerate}[\rm(1)]
\item
\label{it:sign-even-1}%
$C_0$ is either $\CR$ or~$L=\Infl_{G/H}^G(\CR_{\sgn})$ the sign representation of \Cref{Lem:sign-perm-even},
\item
\label{it:sign-even-2}%
$C_1\neq 0$ is free,
\item
\label{it:sign-even-3}%
each $C_j$ is sign-permutation, and
\item
\label{it:sign-even-4}%
each $C_j$ for $j\ge m$ is permutation.
\end{enumerate}

The claim $(\Delta)_0$ is the theorem. We proceed by descending induction on~$m$. First, let $C=D\potimes{G/H}$ be the tensor-induced of the complex for~$H$ as in \Cref{Rec:tens-ind-complexes}. Note that here $n=2$ and $S_n=C_2\cong G/H$ and we have tacitly chosen representatives, say $1,g$, of~$G/H=\{[1],[g]\}$. Since $C$ is a bounded complex, let $m$ be big enough so that~$C_j=0$ for $j\ge m$.

We will now establish the base of induction for this~$m$. Condition~\eqref{it:sign-even-4} is trivially true. And condition~\eqref{it:sign-even-1} is easy: $C_0=D_0\potimes{n}=\CR\potimes{2}\cong\CR$ and $S_n$ acts trivially as everything comes from even degree (zero). For~\eqref{it:sign-even-2} and~\eqref{it:sign-even-3}, given permutation bases $B_j$ of $D_j$ over~$H$, consider for every degree~$d$ the following $\CR$-basis of~$C_d$:
\[
A_d=\mathop{\sqcup}_{0\leq j\leq d}\,B_j\otimes B_{d-j}\subseteq C_d.
\]
This set is preserved by the (diagonal) action of~$H$. For the $G$-action, we only need to check what $g$ does and it acts via~$[g]\in G/H\cong C_2$ as the swap of factors with Koszul sign rule.
It follows that $A_d$ is a sign-permutation basis for $C_d$.
When $d=1$, the only possible $j$ are $0$ and~$1$ and then one of $j$ or $d-j$ is even, so no sign appears. One can check, as in \Cref{Exa:tens-ind-free}, that $i^*(A_1)$ is a free~$G$-set, that is, $C_1=\CR(A_1)$ is a free~$\CRG$-module.
This establishes the base for induction for $m\gg0$.

We will now prove the induction step $(\Delta)_{m+1}\Rightarrow (\Delta)_m$ with $m\geq 0$. We can inflate from $G/H$ the acyclic complex $0\to \CR \to \CR(G/H) \to \CR_{\sgn}\to 0$ to obtain the following quasi-isomorphism~$s$ of complexes of~$\CRG$-modules:
\[
\xymatrix@R=1.5em{
\ C':= \ar@<-.5em>[d]_-{s}
&& \cdots 0 \ar[r]
& \CR \ar[r] \ar[d]
& \CR(G/H) \ar[r] \ar[d]
& 0\cdots
\\
L=
&& \cdots 0 \ar[r]
& 0 \ar[r]
& L \ar[r]
& 0 \cdots
}
\]
Note that $C'$ is a complex of permutation modules. Let $C$ be a complex as in $(\Delta)_{m+1}$ and consider the sign-permutation module $C_m$ in degree~$m$. By \Cref{Lem:sign-perm-even}, we know that $C_m$ decomposes as $C_m=C_m^+\oplus (L\otimes C_m^-)$, where both $C_m^+$ and $C_m^-$ are permutation modules. (To be picky, in the border case of $m=1$ take $C_1^+=C_1$ and $C_1^-=0$, not some fantasy.) There exists a quasi-isomorphism~$t$ as follows
\[
\xymatrix@C=1.8em@R=1.5em{
\ C'':= \ar@<-.5em>[d]_-{t}
& \cdots \ar[r]
& C_{m+2} \ar[r] \ar[d]
& C_{m+1} \ar[r] \ar[d]
& C_{m}^+ \ar[r] \ar[d]
& 0 \ar[r] \ar[d]
& \cdots
\\
C''':=
& \cdots \ar[r]
& 0 \ar[r]
& L\otimes C_m^- \ar[r]
& C_{m-1} \ar[r]
& C_{m-2} \ar[r]
& \cdots
}
\]
whose cone is our acyclic complex~$C$. Observe how the degree~$m$ part is split between the permutation part (top) and the part twisted by~$L$ (bottom). But those two quasi-isomorphisms of bounded complexes of $\CRG$-modules~$s$ and~$t$ only involve free~$\CR$-modules, hence tensoring them yields a quasi-isomorphism~$s\otimes t\colon C'\otimes C''\to L\otimes C'''$. Since both $C'$ and~$C''$ are already complexes of permutation modules, so is~$C'\otimes C''$. Note that at the bottom, $L\otimes C'''$ now admits a permutation module in degree~$m$, namely the module $L\otimes L\otimes C_m^-\cong C_m^-$. Let $C^{\textrm{new}}=\cone(s\otimes t)$. Then this complex is acyclic, consists of sign-permutation modules, and has permutation modules in degrees $\geq m$. It is clear that $C_0^{\textrm{new}}\cong L\otimes C_0$ is still~$\CR$ or~$L$ and it is easy to see that $C_1^{\textrm{new}}$ remains free by Frobenius (\Cref{Rem:tensor}).
Thus $C^{\textrm{new}}$ witnesses the truth of $(\Delta)_m$ and this completes the proof.
\end{proof}

We finish this section with some consequences of \Cref{Thm:1-resol} and its proof. Recall from \Cref{Def:A(G)} the subcategories $\AcatGR$ and~$\BcatGR$ of~$\Db(\CRG)$.

\begin{Prop}
\label{Prop:1-resol}%
The following are equivalent:
\begin{enumerate}[\rm(i)]
\item
\label{it:1-resol-i}%
The trivial $\CRG$-module $\CR$ belongs to~$\AcatGR$.
\item
\label{it:1-resol-ii}%
It admits a $0$-free permutation resolution (\Cref{Lem:beware}):
There is a resolution $0\to P_n\to \cdots \to P_1\to P_0\to \CR\to 0$ by permutation $\CRG$-modules with $P_0$~free.
\end{enumerate}
\end{Prop}

\begin{proof}
\eqref{it:1-resol-i}$\To$\eqref{it:1-resol-ii} is \Cref{Lem:beware}.
For \eqref{it:1-resol-ii}$\To$\eqref{it:1-resol-i}, let $s:P\to \CR$ be a $0$-free permutation resolution and fix $m\geq 1$.
We claim that $s\potimes{m}:P\potimes{m}\to \CR\potimes{m}\cong\CR$ is an $(m-1)$-free permutation resolution.
Indeed, since the modules are all $\CR$-flat, $s\potimes{m}$ is a quasi-isomorphism.
Then $P\potimes{m}$ remains a complex of permutation modules.
(See \Cref{Rem:tensor}.)
And if $\ell\le m-1$ then every summand $P_{\ell_1}\otimes\cdots\otimes P_{\ell_m}$ of $(P\potimes{m})_\ell$ with~$\ell=\ell_1+\cdots+\ell_m$ is free since at least one of the $\ell_i$ must satisfy $\ell_i\le 0$.
\end{proof}

\begin{Rem}
\label{Rem:1-resol}%
A similar characterization of when $\CR$ belongs to~$\BcatGR$ holds with a $0$-projective \ppermutation resolution in~\eqref{it:1-resol-ii}, with the same proof.
\end{Rem}

\begin{Cor}
\label{Cor:R-Acat-p}%
Let $G$ be a $p$-group for some prime~$p$. Then $\CR\in\AcatGR$.
\end{Cor}

\begin{proof}
The criterion of \Cref{Prop:1-resol}\,\eqref{it:1-resol-ii} is precisely \Cref{Thm:1-resol}.
\end{proof}

\begin{Cor}
\label{Cor:R-Acat}%
Let $G$ be an arbitrary finite group. Then we have $\CR\in\AcatGR^\natural$.
\end{Cor}

\begin{proof}
Reduce to Sylow subgroups by \Cref{Cor:Acat-Sylow} and then use \Cref{Cor:R-Acat-p}.
\end{proof}

\begin{Rem}
\label{Rem:R-Acat-abelian}%
If $G$ is abelian, then $\CR\in\AcatGR$ as well. See~\cite[Proposition~7]{balmer-benson:permutation-resolutions}.
\end{Rem}

To end this section, we note a generalization of \Cref{Lem:sign-perm-odd} when $\CR$ is a field (\cf \cite[Theorem~1]{dress:monomial}).
\begin{Lem}
\label{Lem:sign-perm-field}%
Let $\kk$ be a field of positive characteristic~$p$ and $G$ an arbitrary finite group.
Every sign-permutation $\kkG$-module~$M$ (\Cref{Def:sign-perm}) is $p$-permutation.
\end{Lem}

\begin{proof}
If~$p=2$ then $-1=1$ and $M$ is permutation. If $p$ is odd, $M$ restricts to a permutation module over a $p$-Sylow, by \Cref{Lem:sign-perm-odd}. Then use \Cref{Rec:permutation}.
\end{proof}

\begin{Cor}
\label{Cor:k-Bcat}%
Let $G$ be a finite group and assume that $\kk$ is a field of positive characteristic.
Then $\kk\in\Bcat{G}{\kk}$.
\end{Cor}
\begin{proof}
Consider the Koszul complex $\Kos(G;\kk)$ of \Cref{Exa:tens-ind-H=1}.
We have seen in \Cref{Lem:sign-perm-field} that $\Kos(G;\kk)$ defines a $0$-free $p$-permutation resolution of $\kk$.
We conclude by \Cref{Rem:1-resol}.
\end{proof}


\section{Complexes of permutation modules in the derived category}
\label{sec:im(F)}


Recall the thick subcategory of~$\Db(\CRG)$ from \Cref{sec:A(G)} (\Cref{Rem:where-to}):
\begin{equation}
\label{eq:A(G)}%
\AcatGR^\natural=\BcatGR^\natural=\left\{ X\ \bigg|\ {X\oplus\Sigma X\textrm{ admits $m$-free permutation}\atop \textrm{resolutions (Def.\,\ref{Def:perm-resol}) for all }m\ge 0}\ \right\}.
\end{equation}
Our main goal in this section is to prove that the canonical functor~$\bar{\Upsilon}$ of~\eqref{eq:Kb->Db} in the Introduction is fully-faithful, with $\AcatGR^\natural$ as essential image (\Cref{Thm:Kperm->A(G)}).
Moreover, we will prove that~$\AcatGR^\natural=\Db(\CRG)$ for $\CR$ a field (\Cref{Cor:field-case}) and for $\CR$ regular (\Cref{Sch:regular}).
But first, let us verify that $\bar{\Upsilon}$ does at least land inside~$\AcatGR^\natural$.
This is a consequence of our work in the previous section.
\begin{Cor}
\label{Cor:Kbperm-Acat}%
Let $X$ be a bounded complex of permutation modules viewed as an object of~$\Db(\CRG)$.
Then $X\in\AcatGR^\natural$.
\end{Cor}
\begin{proof}
As $\AcatGR^\natural$ is triangulated, it suffices to show $\Ind^G_H(\CR)\in\AcatGR^\natural$ for each subgroup~$H\le G$.
By stability under induction (\Cref{Prop:Ind-Res}), it suffices to show that $\CR\in\Acat{H}{\CR}^\natural$, which we did in \Cref{Cor:R-Acat}.
\end{proof}

\begin{Thm}
\label{Thm:Kperm->A(G)}%
Let $G$ be a finite group and~$\CR$ a commutative noetherian ring.
The canonical functor $\bar{\Upsilon}$ of~\eqref{eq:Kb->Db} restricts to an exact equivalence
\begin{equation}
\label{eq:Kb->A(G)}%
\bar{\Upsilon}\colon\left(\frac{\Kb(\perm(G;\CR))}{\Kbac(\perm(G;\CR))}\right)^\natural\xto{\ \simeq\ }\AcatGR^\natural.
\end{equation}
So $\bar{\Upsilon}$ of~\eqref{eq:Kb->Db} is fully faithful and its essential image in~$\Db(\CRG)$ is~$\AcatGR^\natural$.
\end{Thm}

\begin{proof}
By \Cref{Cor:Kbperm-Acat}, the canonical functor $\Kb(\perm(G;\CR))\to \Db(\CRG)$ lands inside $\AcatGR^\natural$. Since the latter is idempotent-complete there exists a well-defined exact functor~$\bar{\Upsilon}$ as in~\eqref{eq:Kb->A(G)}. By \Cref{Def:perm-resol} (for $m=0$), it is clear that $\bar{\Upsilon}$ is surjective-up-to-direct-summands, hence it suffices to prove that the functor
\[
\bar{\cat{K}}:=\frac{\Kb(\perm(G;\CR))}{\Kbac(\perm(G;\CR))}\xto{\ \bar{\Upsilon}\ }\Db(\CRG)
\]
is fully faithful. As every $X$ is a retract of $X\oplus \Sigma X$, it suffices to prove that the homomorphism $\bar{\Upsilon}\colon \Hom_{\bar{\cat{K}}}(X\oplus \Sigma X,Y)\to \Homcat{\Db(\CRG)}(X\oplus \Sigma X,Y)$ is an isomorphism for every~$X,Y\in\Kb(\perm(G;\CR))$.
Again by \Cref{Cor:Kbperm-Acat}, we know that such a complex $X$ belongs to~$\AcatGR^\natural$ hence $X\oplus\Sigma X\in\AcatGR$. So it suffices to show that for every $X,Y\in \Kb(\perm(G;\CR))$ such that $X\in \AcatGR$, the homomorphism
\[
\bar{\Upsilon}\colon \Hom_{\bar{\cat{K}}}(X,Y)\to \Homcat{\Db(\CRG)}(X,Y)
\]
is a bijection. For surjectivity, let $fs\inv\colon X\to Y$ be represented by a fraction $X\xfrom{s}Z\xto{f}Y$ in $\Cb(\CRG)$ where $s$ is a quasi-isomorphism. Since $X$ belongs to~$\AcatGR$, so does~$Z$ by \Cref{Prop:perm-resol}. So $Z$ admits a $0$-free permutation resolution, \ie there exists a quasi-isomorphism $t\colon P\to Z$ with $P\in \Cb(\perm(G;\CR))$. Hence our morphism $f\,s\inv=(ft)(st)\inv$ comes from $X\xfrom{st}P\xto{ft}Y$ in $\Hom_{\bar{\cat{K}}}(X,Y)$. Injectivity is similar (or follows from conservativity and fullness of~$\bar{\Upsilon}$).
\end{proof}

\begin{Cor}
\label{Cor:field-case}%
Let $\kk$ be a field. Then $\Acat{G}{\kk}^\natural=\Db(\kkG)$ and~\eqref{eq:Kb->Db} is an equivalence.
If moreover $G$ is a $p$-group where $p=\chara(\kk)$ then $\Acat{G}{\kk}=\Db(\kkG)$.
\end{Cor}

\begin{proof}
For the first statement, by \Cref{Cor:Acat-Sylow}, we can assume that $G$ is a $p$-group.
If $p$ is invertible in~$\kk$, we are done by \Cref{Cor:A(G)-R-perfect}.
So it suffices to prove the second statement.
If $\chara(\kk)=p$ and $G$ is a $p$-group then $\kk\in\Acat{G}{\kk}$ by \Cref{Cor:R-Acat-p} and $\kk$ generates $\Db(\kkG)$ as a triangulated category (see \Cref{Rem:p-gp}).
\end{proof}

We have all the ingredients to extend \Cref{Cor:field-case} to regular coefficients.

\begin{Sch}
\label{Sch:regular}%
Let us see that $\AcatGR^\natural=\Db(\CRG)$ when $\CR$ is regular. Together with \Cref{Thm:Kperm->A(G)} this implies \Cref{Thm:regular-intro} in the Introduction.

Pick $X\in\Db(\CR G)$ and let us show that $X$ belongs to~$\AcatGR^\natural$. As before, we reduce to the case of $G$ a $p$-group by \Cref{Cor:Acat-Sylow}, for some prime~$p$. Note that $X$ is automatically $\CR$-perfect since $\CR$ is regular and therefore, by \Cref{Cor:inverting-p}, there exists an exact triangle
\[
P \to X\oplus\Sigma X\to T \to \Sigma P
\]
in~$\Db(\CR G)$ where $P$ is a complex of permutation modules, and $T$ is $p$-torsion. In particular $P\in\Acat{G}{\CR}^\natural$ already, by \Cref{Cor:Kbperm-Acat}. So it suffices to prove $T\in\Acat{G}{\CR}^\natural$, \ie we can assume that $p^n\cdot X=0$ for some $n\gg1$. By \Cref{Rem:torsion} we then have $X\in\thick(\cone(X\xto{p}X))$. But $\cone(X\xto{p}X)\cong\cone(\CR\xto{p}\CR)\Lotimes_{\CR}X\cong i_*\rmL i^*(X)$ by the projection formula for the adjunction
\[
\xymatrix@R=1.3em{
\Db(\CR G) \ar@<-.3em>[d]_-{\rmL i^*=\bar{\CR}\Lotimes_{\CR}-}
\\
\Db(\bar{\CR} G) \ar@<-.3em>[u]_-{i_*}
}
\]
given by the usual extension and restriction of scalars along $\CR\to \bar{\CR}:=\CR/p$. Hence
\begin{equation}
\label{eq:aux-Z-1}%
X\in \thick(i_*(\Db(\bar{\CR} G))).
\end{equation}
We claim that the image of $\Db(\bar{\CR})$ under $\Infl_1^G\colon \Db(\bar{\CR})\to \Db(\bar{\CR}G)$ generates the whole of~$\Db(\bar{\CR}G)$ as a thick subcategory. This uses nilpotence of the augmentation ideal $I:=\Ker(\bar{\CR}G\to \bar{\CR})$, and the associated finite filtration  $\cdots I^{\ell+1} N\subseteq I^\ell N\cdots$  of any $\bar{\CR}G$-module~$N$, in which every $I^\ell N/I^{\ell+1} N$ has trivial $G$-action.
Then one uses regularity of~$\CR$ one more time and the commutativity of the following square
\begin{equation}
\label{eq:aux-Z-3}%
\vcenter{\xymatrix{
\Db(\bar\CR) \ar[r]^-{i_*} \ar[d]_-{\Infl_1^G}
& \Db(\CR) =\Dperf(\CR)=\thick(\CR) \kern-10em
\ar[d]^-{\Infl_1^G}
\\
\Db(\bar\CR G) \ar[r]^-{i_*}
& \Db(\CR G)
}}
\end{equation}
to continue from~\eqref{eq:aux-Z-1} and deduce
\[
X\in\thick(i_*(\Infl_1^G\Db(\bar\CR)))\underset{\textrm{(\ref{eq:aux-Z-3})}}{\,\subseteq\,}\thick(\Infl_1^G\CR)\subseteq\AcatGR^\natural
\]
where the last inclusion holds because $\CR\in\AcatGR^\natural$ by \Cref{Cor:R-Acat}.
\end{Sch}

\begin{Rem}
\label{Rem:Carlson}%
Using \Cref{Sch:regular} and Carlson~\cite{carlson:induction-abelem} one can easily show a form of Chouinard's theorem for~$\AcatGR^\natural$, namely if a complex $X\in \Db(\CRG)$ is such that $\Res^G_E(X)$ belongs to~$\Acat{E}{\CR}^\natural$ for all elementary abelian subgroups~$E\le G$, then $X\in\AcatGR^\natural$.
Indeed, by~\cite{carlson:induction-abelem}, the trivial module $\bbZ$ belongs to the thick subcategory of $\Db(\bbZ G)$ generated by modules induced from elementary abelian subgroups.
Hence $X\cong \bbZ\Lotimes_{\bbZ} X$ belongs to the thick subcategory of $\Db(\CRG)$ generated by complexes in $\Ind^G_E(\Db(\bbZ E)\Lotimes_{\bbZ}\Res^G_EX)$ (use
the projection formula for $\Ind_E^G\adj\Res^G_E$). We then conclude from $\Acat{E}{\bbZ}^\natural=\Db(\bbZ E)$ by \Cref{Sch:regular}.
\end{Rem}


\section{Density and Grothendieck group}
\label{sec:dense-and-K_0}%


We want to use Thomason's classification of dense subcategories to derive consequences from the results of previous sections. Let us remind the reader.

\begin{Rec}
\label{Rec:Thomason-dense}%
Given an essentially small triangulated category $\cT$ we may consider its Grothendieck group, $\rmK_0(\cT)$, the free abelian group generated by isomorphism classes of objects in $\cT$ quotiented by the relation $[X]+[Z]=[Y]$ for each exact triangle $X\to Y\to Z\to \Sigma X$ in~$\cT$. In particular $-[X]=[\Sigma X]$.

For each dense triangulated subcategory $\cA\subseteq \cT$ (\Cref{Rec:dense}) the map $\rmK_0(\cA)\to \rmK_0(\cT)$ is injective (\cite[Corollary~2.3]{thomason:classification}) hence $\rmK_0(\cA)$ defines a subgroup of~$\rmK_0(\cT)$. Conversely, each subgroup $S\subseteq \rmK_0(\cT)$ defines a dense subcategory
\[
\cA(S):=\SET{X\in \cT}{[X]\in S\textrm{ in }\rmK_0(\cT)}.
\]
By~\cite[Theorem~2.1]{thomason:classification} these constructions yield a well-defined bijection
\[
\{\text{ dense triangulated subcategories of }\cT\ \}\overset{\sim}{\,\longleftrightarrow\,}\{\text{ subgroups of }\rmK_0(\cT)\ \}.
\]
\end{Rec}

Recall the triangulated subcategories $\AcatGR$ and~$\BcatGR$ of \Cref{Def:A(G)} and recall that $\Img(\bar{\Upsilon})=\AcatGR^\natural=\BcatGR^\natural$ by \Cref{Thm:Kperm->A(G)} and \Cref{Prop:A-dense-in-B}.

\begin{Cor}
\label{Cor:Omega(M)}%
Let $M\in\mmod{\CRG}$ and $\Omega(M)=\Ker(\CRG\otimes M\onto M)$. Suppose that $M$ belongs to~$\Img(\bar{\Upsilon})$.
Then $M\oplus \Omega(M)$ belongs to~$\BcatGR$.
If furthermore $M$ is~$\CR$-free then $M\oplus\Omega(M)$ belongs to~$\AcatGR$.
\end{Cor}

\begin{proof}
We need to prove that in the group~$\rmK_0(\BcatGR^\natural)$ the class $[M\oplus\Omega(M)]=[\CRG\otimes M]$ belongs to the subgroup $\rmK_0(\BcatGR)$.
Since projective modules belong to~$\BcatGR$, we have $\Dperf(\CRG)\subseteq\BcatGR$. So it suffices to show that $\CRG\otimes M$ is perfect over~$\CRG$.
But $\Res^G_1(M)$ is perfect (\Cref{Cor:A(G)-R-perfect}) and by Frobenius $\CRG\otimes M\cong\Ind_1^G\Res^G_1 M\in\Dperf(\CRG)$. Hence the first claim.
Similarly, if $M$ is moreover $\CR$-free then $\CRG\otimes M$ is $\CRG$-free and $[M\oplus \Omega(M)]\in \rmK_0(\AcatGR)$.
\end{proof}

We record the following statement for later use in~\cite{balmer-gallauer:resol-big}:
\begin{Cor}
\label{Cor:M-resol-lim}%
Let $M$ be an $\CRG$-module that belongs to~$\Img(\bar{\Upsilon})$. Then there exists a sequence of quasi-isomorphisms of bounded complexes in~$\Chain_{\geq 0}(\CRG)$
\[
\cdots \to Q(n+1)\to Q(n)\to \cdots \to Q(1)\to M\oplus \Omega(M)
\]
such that $Q(n)$ consists of \ppermutation $\CRG$-modules, and in the range $0\le d<n$, the module $Q(n)_d$ is projective and $Q(n+1)_d\to Q(n)_d$ is the identity. In particular, the sequence $\cdots \to Q(n)\to \cdots \to Q(1)$ is eventually stationary in each degree and $P=\lim_{n\to \infty} Q(n)$, computed degreewise, is a projective resolution of~$M\oplus \Omega(M)$.
\end{Cor}

\begin{proof}
By \Cref{Cor:Omega(M)} we can apply \Cref{Prop:M-resol-lim} to $M\oplus \Omega(M)$.
\end{proof}

We now turn our attention to the case of a field~$\kk$ of positive characteristic~$p$.

\begin{Rem}
\label{Rem:K_0-field-case}%
We want to apply Thomason's Theorem to the triangulated subcategories $\AcatGk$ and~$\BcatGk$ of $\cat{T}=\Db(\kkG)$ introduced in~\Cref{Def:A(G)}, that are dense by \Cref{Cor:field-case}. The Grothendieck group of~$\Db(\kkG)$ as a triangulated category coincides with the Grothendieck group of~$\mmod{\kkG}$ as an abelian category
\[
\rmK_0(\Db(\mmod{\kk G}))\cong \rmK_0(\mmod{\kk G})=\rmG_0(\kk G).
\]
This Grothendieck group is free abelian on the set of isomorphism classes of simple $\kk G$-modules.
\,(\footnote{\,In the classic reference~\cite{serre:linear-reps-finite-groups}, Serre writes $R_{\kk}(G)$ for $\KG{G}{\kk}$ and $P_{\kk}(G)$ for the Grothendieck group $\PGk$ of projective modules.
We do \emph{not} adopt this notation to avoid confusion with our coefficient ring~$\CR$ and the category~$\AcatGk$.})
We have the inclusions of subgroups in $\rmG_0({\kkG})$:
\[
\xymatrix@R=1.5em{
\bbZ\cdot[\kk G] \ar@{}[r]|{\subseteq} \ar@{}[d]|-{\rotatebox[origin=c]{270}{$\scriptstyle\subseteq$}}
& \PGk=\rmK_0(\proj(\kkG)) \ar@{}[d]|-{\rotatebox[origin=c]{270}{$\scriptstyle\subseteq$}}
\\
\AGk:=\rmK_0(\AcatGk)\ar@{}[r]|-{\subseteq}
& \BGk:=\rmK_0(\BcatGk)\ar@{}[r]|-{\subseteq}
& \rmG_0({\kkG}).
}
\]
Injectivity of $\AGk\to \RGk$ and $\BGk\to \RGk$ follows from density (\Cref{Cor:field-case}) and \Cref{Rec:Thomason-dense}. All inclusions displayed above are then straightforward, already for the corresponding categories.

The quotient $\RGk/\PGk$ is well-known to be a finite abelian $p$-group, whose exponent is a power of~$p$ dividing~$|G|$. See \cite[\S\,16, Theorem~35]{serre:linear-reps-finite-groups}\,(\footnote{\,The general assumptions of~\cite[p.\,115]{serre:linear-reps-finite-groups} hold for any~$\kk$ by~\cite[Theorem, p.~23]{hochster:complete-local-rings}.}).
Hence the same is true for $\RGk/\BGk$ but we shall prove more in \Cref{Cor:KQ=G_0}, namely that $\BGk=\RGk$.

The subgroup~$\AGk\subseteq\RGk$ is not of finite index in general, simply because permutation modules are defined integrally. The subgroup $\AGk$ is always contained in the image of~$\KG{G}{\bbF_{\!p}}$ inside~$\RGk$ that can have infinite index when $\kk G$ has simple modules not defined over~$\bbF_{\!p}$.
\end{Rem}

\begin{Exa}
The cokernel of the `Cartan homomorphism' $\KP{G}{\CR}\to\KG{G}{\CR}$ is not always of finite exponent when $\CR$ is not a field, even for a discrete valuation ring. Take $\CR=\bbZ_{(2)}$ and $G=C_2$ cyclic of order~2. Then $\KP{C_2}{\CR}=\bbZ\cdot[\CR C_2]$ since $\CR C_2$ is local. Let $\CR^+=\CR$ with trivial $C_2$-action. Rationally, in $\KG{C_2}{\bbQ}\cong \bbZ\cdot[\bbQ^+]\oplus \bbZ\cdot[\bbQ^-]$ for $\bbQ^+$ trivial and~$\bbQ^-=\bbQ$ with sign action, our $[\CR^+]$ maps to~$[\bbQ^+]$ but $[\CR C_2]$ maps to~$[\bbQ^+]+[\bbQ^-]$.
So no non-zero multiple of $[\CR^+]\in\KG{C_2}{\CR}$ belongs to~$\KP{C_2}{\CR}$.
\end{Exa}

Applying Thomason's classification (\Cref{Rec:Thomason-dense}) to the situation of \Cref{Rem:K_0-field-case} we get for instance:

\begin{Cor}
\label{Cor:resolution-K-class}%
An $M\in\mmod{\kk G}$ admits $m$-free permutation resolutions for all $m\ge 0$ if and only if its class $[M]\in \RGk$ belongs to the subgroup $\AGk$.
\qed
\end{Cor}

We do not have a description of~$\AGk$ in general but it is already remarkable to have a condition in terms of the class of~$M$ in the Grothendieck group. Using only that free modules belong to $\AcatGk$ we get some interesting consequences.

\begin{Cor}
\label{Cor:complement1}%
Let $M\in\mmod{\kk G}$ and consider $\Omega(M)=\Ker(\kkG\otimes M\onto M)$. Then $M\oplus\Omega(M)$ admits $m$-free permutation resolutions for all~$m\ge 0$. In particular, $M\oplus\Omega(M)$ admits a finite resolution by finitely generated permutation modules.
\end{Cor}

\begin{proof}
The first part follows from \Cref{Cor:Omega(M)}, since every~$M$ is~$\kk$-free. The second part follows from \Cref{Cor:0-resolution}.
\end{proof}

\begin{Rem}
\label{Rem:Kb-Db-dense}%
The essential images of the obvious functors $\Kb(\perm(G;\kk))\to \Db(\kk G)$ and $\Kb(\perm(G;\kk)^\natural)\to \Db(\kk G)$ are also dense triangulated subcategories.
Indeed, these essential images are the same as those of the functors
\[
\frac{\Kb(\perm(G;\kk))}{\Kbac(\perm(G;\kk))}\to \Db(\kk G)
\qquadtext{and}
\frac{\Kb(\perm(G;\kk)^\natural)}{\Kbac(\perm(G;\kk)^\natural)}\to \Db(\kk G).
\]
As these functors are full (and faithful) by \Cref{Thm:Kperm->A(G)}, their images are triangulated subcategories. As these images contain $\AcatGk$, they are dense by \Cref{Cor:field-case}. In fact, the right-hand functor is already essentially surjective, as we shall see in \Cref{Thm:kG-pperm-resolutions}. By Thomason, it suffices to understand what happens on~$\rmK_0$.
\end{Rem}

\begin{Prop}[Boltje/Bouc]
\label{Prop:Bouc-Boltje}%
The canonical homomorphism
\[
\rmK_0(\perm(G;\kk)^\natural)\to \RGk
\]
from the additive Grothendieck group of $p$-permutation modules {\rm(}\aka the $p$-permu\-ta\-tion ring~\cite{bouc-thevenaz:p-perm-ring}, or trivial source ring~\cite{boltje:trivial-source}{\rm)} is a surjection onto~$\RGk$. \textrm{\rm(\footnote{\,This homomorphism is rarely injective: Even for $G=C_p$ we get $\bbZ[x]/(x^2-px)\to\bbZ$.})}
\end{Prop}

\begin{proof}
Brauer's Theorem in the modular case~\cite[\S\,17.2]{serre:linear-reps-finite-groups} asserts that
\[
\Ind:\oplus_H \KG{H}{\kk}\onto \RGk
\]
is surjective, where $H$ runs through the so-called $\Gamma_K$-elementary subgroups of~$G$ (with notation of \cite[\S\,12.4]{serre:linear-reps-finite-groups}). So it suffices to prove the result for~$G$ of that type. In that case, we prove that every \emph{simple} $\kkG$-module~$M$ is $p$-permutation.

In the `easy case' where $G=C\rtimes Q$ with $C$ (cyclic) of order a power of~$p$ and $Q$ of order prime to~$p$, we can consider the non-zero submodule $M^C$ of~$M$. As $C$ is normal in~$G$, it follows that $M^C$ is a $\kkG$-submodule of~$M$, hence equal to it. In short, $M$ has trivial restriction to the $p$-Sylow~$C$ of~$G$, hence is $p$-permutation.

The `tricky case' is when $G=C\rtimes P$ where $P$ is a $p$-Sylow and $C$ is cyclic of order~$m$ prime to~$p$. Using induction on~$|G|$ as in the proof of~\cite[\S\,17.3, Theorem~41]{serre:linear-reps-finite-groups}, we reduce to the case where $M$ is a finite extension $\kk'=\kk[X]/f(X)$ where $f$ is an irreducible factor of $X^m-1$, on which $P$ acts through $k$-automorphisms of the field $\kk'$. As $m$ and $p$ are coprime, the cyclotomic extension $\kk'/\kk$ is separable and hence Galois. It follows from the normal basis theorem that the $\kk[P]$-module $\kk'$ is permutation, and we conclude as before (see \Cref{Rec:permutation}).
\end{proof}

\begin{Rem}
\label{Rem:Bouc}%
If we assume the field $\kk$ `sufficiently large' (\cf \cite[p.\,115]{serre:linear-reps-finite-groups}), \eg\ algebraically closed, the above $\Gamma_K$-elementary subgroups are $q$-elemen\-tary for a prime~$q$, \ie of the form $C\times Q$ for a $q$-group~$Q$ and a cyclic group~$C$ of order prime to~$q$. In that case, the $p$-Sylow of~$G$ is normal and we can apply the `easy case' of the above proof. (In an early version of this paper, we assumed $\kk$ sufficiently large for that reason.) The `tricky case' was communicated to us by Serge Bouc.
\end{Rem}

\begin{Rem}
\label{Rem:Boltje}%
Robert Boltje gave us a different argument to remove `$\kk$ sufficiently large' in \Cref{Prop:Bouc-Boltje}, based on a natural section of the homomorphism~$\rmK_0(\perm(G;\kk)^\natural)\to \RGk$. This uses the canonical induction formula for the Brauer character ring, as well as Galois descent to reduce to the case of~$\kk$ sufficiently large discussed above.
\end{Rem}

\begin{Cor}
\label{Cor:KQ=G_0}%
We have $\BGk=\RGk$ hence $\BcatGk=\Db(\kkG)$.
\end{Cor}

\begin{proof}
First, we claim that the subgroup $\BGk\subseteq\RGk$ is an ideal.
Indeed, by \Cref{Prop:Bouc-Boltje}, it suffices to show that $\BGk$ is closed under multiplying by the class of a $p$-permutation module, which is straightforward.
We are therefore reduced to show that $1=[\kk]$ belongs to this ideal~$\BGk$.
This is true by \Cref{Cor:k-Bcat}.
The second statement follows by Thomason (\Cref{Rec:Thomason-dense}).
\end{proof}

Summarizing the situation, we have our main result:
\begin{Thm}
\label{Thm:kG-pperm-resolutions}%
Let $G$ be a finite group, and $\kk$ a field of characteristic $p>0$.
Then every $\kkG$-module admits a finite $p$-permutation resolution.
Moreover, the canonical functor (see \Cref{Rec:permutation})
\[
\frac{\Kb(\perm(G;\kk)^\natural)}{\Kbac(\perm(G;\kk)^\natural)}\isoto \Db(\kk G)
\]
is an equivalence. So the left-hand quotient is already idempotent-complete.
\end{Thm}

\begin{proof}
Every module $M\in\mmod{\kkG}$ belongs to~$\BcatGk$ by \Cref{Cor:KQ=G_0}. It follows that $M$ admits a $p$-permutation resolution by \Cref{Cor:0-resolution}.
For the equivalence, we resume the discussion of \Cref{Rem:Kb-Db-dense}. By \Cref{Thm:Kperm->A(G)} and \Cref{Cor:field-case}, the quotient ${\Kb(\perm(G;\kk)^\natural)}/{\Kbac(\perm(G;\kk)^\natural)}$ is a dense subcategory of~$\Db(\kkG)$.
As every object in $\BcatGk$ is quasi-isomorphic to a complex of \ppermutation modules, this dense subcategory contains $\BcatGk$.
By \Cref{Cor:KQ=G_0}, it therefore coincides with $\Db(\kkG)$.
\end{proof}


\end{document}